\documentclass[12pt,english,reqno]{amsart}
\usepackage{a4wide}
\usepackage{amsmath,amssymb,amsfonts,amsthm}

\usepackage[
  pdfauthor={Bernd C. Kellner},
  pdfkeywords={Erdos-Moser equation, consecutive values of polynomials},
  pdftitle={On stronger conjectures that imply the Erdos-Moser conjecture},
  linktocpage,colorlinks,bookmarksnumbered]{hyperref}

\DeclareMathOperator{\denom}{denom}
\DeclareMathOperator{\ord}{ord}

\newcommand{\pdiv}{\mid}
\newcommand{\pdiveq}{\mid\mid}
\newcommand{\notdiv}{\nmid}

\newcommand{\NN}{\mathbb{N}}

\newtheorem{theorem}{Theorem}
\newtheorem{prop}{Proposition}
\newtheorem{lemma}{Lemma}
\newtheorem{corl}{Corollary}
\newtheorem{conj}{Conjecture}

\theoremstyle{remark}
\newtheorem{remark}{Remark}

\title[On stronger conjectures that imply the Erd\H{o}s-Moser conjecture]
{On stronger conjectures that imply\\ the Erd\H{o}s-Moser conjecture}
\author{Bernd C. Kellner}
\date{}
\address{Mathematisches Institut, Universit\"at G\"ottingen, Bunsenstr.\ 3–-5,
37073 G\"ottingen, Germany}
\email{bk@bernoulli.org}

\begin{document}

\begin{abstract}
The Erd\H{o}s-Moser conjecture states that the Diophantine equation
$S_k(m) = m^k$, where $S_k(m)=1^k+2^k+\cdots+(m-1)^k$, has no solution for
positive integers $k$ and $m$ with $k \geq 2$. We show that stronger
conjectures about consecutive values of the function $S_k$, that seem to be
more naturally, imply the Erd\H{o}s-Moser conjecture.
\end{abstract}

\maketitle

\section{Introduction}

Let $k$ and $m$ be positive integers throughout this paper. Define
\[
   S_k(m) = 1^k + 2^k + \cdots + (m-1)^k.
\]

\begin{conj}[Erd\H{o}s-Moser] \label{conj:em}
The Diophantine equation
\begin{equation} \label{eq:em}
   S_k(m) = m^k
\end{equation}
has only the trivial solution $(k,m) = (1,3)$ for positive integers $k$, $m$.
\end{conj}

In 1953 Moser \cite{Moser:1953} showed that if a solution of \eqref{eq:em}
exists for $k \geq 2$, then $k$ must be even and $m > 10^{10^6}$.
Recently, this bound has been greatly increased to $m > 10^{10^9}$
by Gallot, Moree, and Zudilin \cite{GMZ:2011}. So it is widely believed that
non-trivial solutions do not exist.
Comparing $S_k$ with the integral $\int\! x^k dx$, see \cite{GMZ:2011}, one gets
an easy estimate that
\begin{equation} \label{eq:estim-k-m}
   k < m < 2k.
\end{equation}

A general result of the author \cite[Prop.~8.5, p.~436]{Kellner:2007} states that
\begin{equation} \label{eq:m2-div-Bk}
   m^{r+1} \pdiv S_k(m) \quad \iff \quad m^r \pdiv B_k
\end{equation}
for $r=1,2$ and even $k$, where $B_k$ denotes the $k$-th Bernoulli number.
Thus a non-trivial solution $(k,m)$ of \eqref{eq:em} has the property
that $m^2$ must divide the numerator of $B_k$ for $k \geq 4$;
this result concerning \eqref{eq:em} was also shown in \cite{MRU:1992} in a
different form.

Because the Erd\H{o}s-Moser equation is very special, one can consider
properties of consecutive values of the function $S_k$ in general.
This leads to two stronger conjectures, described in the next sections,
that imply the conjecture of Erd\H{o}s-Moser.

\section{Preliminaries}

We use the following notation. We write $p^r \pdiveq m$ when $p^r \pdiv m$
but $p^{r+1} \notdiv m$, i.e., $r = \ord_p m$ where $p$ always denotes a prime.
Next we recall some properties of the Bernoulli numbers and the function $S_k$.

The Bernoulli numbers $B_n$ are defined by
\[
  \frac{z}{e^z-1} = \sum_{n=0}^\infty B_n \frac{z^n}{n!},
    \quad |z| < 2 \pi.
\]
These numbers are rational where $B_n = 0$ for odd $n > 1$ and
$(-1)^{\frac{n}{2}+1}B_n > 0$ for even $n > 0$. A table of the Bernoulli numbers
up to index 20 are given in \cite[p.~437]{Kellner:2007}.
The denominator of $B_n$ for even $n$ is described by the von Staudt-Clausen
theorem, see \cite[p.~233]{IR:1990}, that
\begin{equation} \label{eq:Bk-denom}
   \denom(B_n) = \prod_{p-1 \pdiv n} p.
\end{equation}

The function $S_k$ is closely related to the Bernoulli numbers
and is given by the well-known formula, cf. \cite[p.~234]{IR:1990}:
\begin{equation} \label{eq:Sk-sum}
  S_k(m) = \sum_{\nu=0}^k \binom{k}{\nu} B_{k-\nu} \frac{m^{\nu+1}}{\nu+1}.
\end{equation}

\section{Stronger conjecture --- Part I}

The strictly increasing function $S_k$ is a polynomial of
degree $k+1$ as a result of \eqref{eq:Sk-sum}. One may not expect that
consecutive values of $S_k$ have highly common prime factors, such that
$S_k(m+1)/S_k(m)$ is an integer for sufficiently large $m$.

\begin{conj} \label{conj:Sk-div}
Let $k, m$ be positive integers with $m \geq 3$. Then
\begin{equation} \label{eq:conj-Sk-div}
   \frac{S_k(m+1)}{S_k(m)} \in \NN
     \quad \iff \quad (k,m) \in \{ (1,3), (3,3) \}.
\end{equation}
\end{conj}

Note that we have to require $m \geq 3$, since $S_k(1) = 0$ and $S_k(2) = 1$
for all $k \geq 1$. Due to the well-known identity $S_1(m)^2 = S_3(m)$,
a solution for $k=1$ implies a solution for $k=3$. Hereby we have the only
known solutions
\begin{equation} \label{eq:Sk-div-sol}
   \frac{1+2+3}{1+2} = 2 \quad \text{and} \quad \frac{1^3+2^3+3^3}{1^3+2^3} = 4
\end{equation}
based on some computer search. Since $S_k(m+1)/S_k(m) \to 1$ as $m \to \infty$,
it is clear that we can only have a finite number of solutions for a fixed $k$.
By $S_k(m+1) = S_k(m) + m^k$, one easily observes that \eqref{eq:conj-Sk-div}
is equivalent to
\[
   a \, S_k(m) = m^k \quad \iff \quad (a,k,m) \in \{ (1,1,3), (3,3,3) \},
\]
where $a$ is a positive integer. This gives a generalization of \eqref{eq:em}.

\begin{prop}
Conjecture \ref{conj:Sk-div} implies Conjecture \ref{conj:em}.
\end{prop}

\begin{proof}
Eq.~\eqref{conj:em} can be rewritten as $2S_k(m)=S_k(m+1)$ after adding
$S_k(m)$ on both sides. Conjecture \ref{conj:Sk-div} states that $S_k(m+1)/S_k(m)$
is not a positive integer except for the cases $(k,m) = (1,3)$ and $(k,m) = (3,3)$
as given in \eqref{eq:Sk-div-sol}. This implies Conjecture \ref{conj:em},
which predicts $S_k(m+1)/S_k(m) \neq 2$ for $k \geq 2$.
\end{proof}

\section{Stronger conjecture --- Part II}

The connection between the function $S_k$ and the Bernoulli numbers leads to
the following theorem, which we will prove later. In the following we always
write $B_k = N_k/D_k$ in lowest terms with $D_k > 0$ for even $k$.
For now we write $(a,b)$ for $\gcd(a,b)$.

\begin{theorem} \label{thm:Sk-gcd}
Let $k, m$ be positive integers with even $k$. Define
\[
   g_k(m) = \frac{(S_k(m),S_k(m+1))}{m}.
\]
Then
\[
   \min_{m \, \geq \, 1} \, g_k(m) = \frac{1}{D_k}
     \quad \text{and} \quad
     \max_{m \, \geq \, 1} \, g_k(m) \geq |N_k|.
\]
Generally
\[
   g_k(m)=1 \quad \iff \quad (D_k N_k,m)=1
\]
and special values are given by
\[
   g_k(D_k) = \frac{1}{D_k}, \quad g_k(|N_k|) = |N_k|,
     \quad \text{and} \quad g_k(D_k \, |N_k|) = |B_k|.
\]
More generally,
\[
   g_k(m) = |N_k|, \quad \text{if\ } (D_k,m)=1
     \text{\ and\ } |N_k| \pdiv m.
\]
In particular if $N_k$ is square-free, then
\[
   g_k(m)=\frac{(N_k,m)}{(D_k,m)}
     \quad \text{and} \quad
     \max_{m \, \geq \, 1} \, g_k(m) = |N_k|.
\]
\end{theorem}

\begin{remark} \label{rem:Nk-prime}
It is well-known that $|N_k| = 1$ exactly for $k \in \{2,4,6,8\}$. Known
indices $k$, where $|N_k|$ is prime, are recorded as sequence A092132 in
\cite{Sloane:2011}: $10, 12, 14, 16, 18, 36, 42$. Sequence A090997 in
\cite{Sloane:2011} gives the indices $k$, where $N_k$ is not square-free:
50, 98, 150, 196, 228, $\ldots$ . By this, all $N_k$ are square-free for
$2 \leq k \leq 48$.
\end{remark}

Since $S_k(m+1) = S_k(m)+m^k$, we have
\begin{equation} \label{eq:loc-gcd-Sk}
   (S_k(m),S_k(m+1)) = (S_k(m),m^k),
\end{equation}
giving a connection with \eqref{eq:em}.
The function $g_k$ heavily depends on the Bernoulli number $B_k$.
For $2 \leq k \leq 48$ and some higher indices $k$ we even have
\[
   \min_{m \, \geq \, 1} \, g_k(m) \, \cdot \,
   \max_{m \, \geq \, 1} \, g_k(m) = |B_k|.
\]

The problem is to find an accurate upper bound of $g_k$ to solve \eqref{eq:em}.
This relation is demonstrated by Theorem~\ref{thm:gk-bound} below and we raise
the following conjecture based on Theorem~\ref{thm:Sk-gcd} and some computations.

\begin{conj} \label{conj:gk-max}
The function $g_k$ has an upper bound as given in Theorem~\ref{thm:gk-bound}.
\end{conj}

\begin{theorem} \label{thm:gk-bound}
Let $k, m, r$ be positive integers with even $k \geq 10$. If
\[
   \max_{m \, \geq \, 1} \, g_k(m) < |N_k| \log^r |N_k|
     \quad \text{for } k \geq C_r
\]
and \eqref{eq:em} has no solution for $k < C_r$, where $C_r$ is an
effectively computable constant, then Conjecture \ref{conj:em} is true.
In particular, one can choose $C_r=10$ for $r=1,\ldots,6$.
\end{theorem}

\begin{proof}
Considering Theorem~\ref{thm:Sk-gcd} and \eqref{eq:loc-gcd-Sk},
a possible solution of \eqref{eq:em} must trivially satisfy
\begin{equation} \label{eq:loc-mk-gcd}
   m^k = (S_k(m),m^k) = m \, g_k(m).
\end{equation}
For $k=2,4,6,8$ there is no solution of \eqref{eq:em}, since $|N_k|=1$.
Now let $k \geq 10$.
Using the relation of $B_k$ to the Riemann zeta function by Euler's formula,
cf. \cite[p.~231]{IR:1990}, we have
\[
   |B_k| = 2 \zeta(k) \frac{k!}{(2\pi)^k}.
\]
Since $\zeta(s) \to 1$ monotonically as $s \to \infty$ and
$\zeta(2)=\pi^2/6$, we obtain
\[
   |N_k| < \frac{\pi^2}{3} \frac{k!}{(2\pi)^k} D_k <
     \frac{2\pi^2}{3} \frac{k!}{\pi^k},
\]
using the fact that $D_k \pdiv 2(2^k-1)$, see \cite{CH:1972}.
Stirling's series of the Gamma function, cf. \cite[p.~481]{GKP:1994},
states that $k! < \sqrt{2\pi k} \, k^k \, e^{-k+1/12k}.$
Since $e^{1/12k} < \frac{11}{10}$, we deduce that
\[
   |N_k| < \eta \, k^\frac32 \left( \frac{k}{e \pi} \right)^{k-1}
     \quad \text{with} \quad
     \eta = \frac{11}{15} \frac{\pi}{e} \sqrt{2\pi} \approx 2.12 .
\]
Further we conclude that $\log |N_k| < k \log (k / \pi)$.
Finally, we achieve that
\begin{equation} \label{eq:loc-estim-Nk-log-r}
   |N_k| \log^r |N_k| <  f_r(k) \left( \frac{k}{e \pi} \right)^{k-1}
\end{equation}
with
\[
   f_r(k) = \eta \, k^{\frac32+r} \log^r (k / \pi).
\]
For a fixed $r$ we have $\sqrt[k-1]{f_r(k)} \to 1$ as $k \to \infty$. Define
\[
    I(r) = \min \left\{ n \geq 10 : \sqrt[k-1]{f_r(k)} < e \pi
      \text{ for all } k \geq n \right\},
\]
which is an increasing function depending on $r$.
A short computation shows that $I(r)=10$ for $r=1,\ldots,6$. We set $C_r=I(r)$.
Consequently \eqref{eq:loc-estim-Nk-log-r} turns into
\begin{equation} \label{eq:loc-estim-sqrt-Nk-log}
   \sqrt[k-1]{|N_k| \log^r |N_k|} < k \quad \text{for } k \geq C_r.
\end{equation}
Now, we assume that \eqref{eq:em} has no solution for $k < C_r$
and that
\begin{equation} \label{eq:loc-estim-gk-log}
   \max_{m \, \geq \, 1} \, g_k(m) < |N_k| \log^r |N_k|
     \quad \text{for } k \geq C_r.
\end{equation}
According to \eqref{eq:loc-mk-gcd}, \eqref{eq:loc-estim-sqrt-Nk-log},
and \eqref{eq:loc-estim-gk-log}, we then achieve that $m < k$ for $k \geq C_r$,
which contradicts \eqref{eq:estim-k-m}. Thus there is no solution of
\eqref{eq:em} for all $k \geq 2$ implying Conjecture~\ref{conj:em}.
\end{proof}

To prove Theorem \ref{thm:Sk-gcd}, we shall need some preparations and a
refinement of \eqref{eq:m2-div-Bk}.

\begin{theorem} \label{thm:congr-sk-bk}
Let $k, m$ be positive integers where $k$ is even and $m \geq 2$. Then
\begin{alignat*}{2}
   S_k(m) &\equiv B_k \, m \pmod{m},
     \quad &&\text{if } k \geq 2, \\
   S_k(m) &\equiv B_k \, m \pmod{m^2},
     \quad &&\text{if } k \geq 4 \text{ and } (D_k,m)=1, \\
   S_k(m) &\equiv B_k \, m \pmod{m^3},
     \quad &&\text{if } k \geq 6 \text{ and } m \pdiv N_k.
\end{alignat*}
More precisely for $p^r \pdiveq m$:
\begin{alignat*}{2}
   S_k(m) &\equiv B_k \, m \pmod{p^{2r}},
     \quad &&\text{if } k \geq 4 \text{ and } p \notdiv D_k, \\
   S_k(m) &\equiv B_k \, m \pmod{p^{3r}},
     \quad &&\text{if } k \geq 6 \text{ and } p \pdiv N_k.
\end{alignat*}
\end{theorem}

\begin{proof}
This follows by exploiting the proof of
\cite[Prop.~8.5, pp.~436-437]{Kellner:2007}.
\end{proof}

\begin{lemma} \label{lem:seq-gcd}
Let $a, b$ be positive integers. The sequence $\{(a,b^\nu)\}_{\nu \geq 1}$
is increasing and eventually constant. If $(a,b^r)=(a,b^{r+1})$ for some
$r \geq 1$, then $\{(a,b^\nu)\}_{\nu \geq r}$ is constant. Especially if
$\ord_p a \leq s \ord_p b$, then $\ord_p \, (a,b^\nu) = \ord_p a$
for $\nu \geq s$.
\end{lemma}

\begin{proof}
If $(a,b)=1$, then $(a,b^\nu)=1$ for $\nu \geq 1$.
Assume that $(a,b) > 1$. For each $p \pdiv (a,b)$,
we have $\ord_p \, (a,b^\nu) = \min \{ \ord_p a, \nu \ord_p b\}$,
which is increasing and bounded as $\nu \to \infty$.
It follows that if $\ord_p a \leq s \ord_p b$, then
$\ord_p \, (a,b^\nu) = \ord_p a$ for $\nu \geq s$.
Considering all primes $p \pdiv (a,b)$, we deduce that
$(a,b^r)=(a,b^{r+1})$ for some $r \geq 1$ implies that
$(a,b^\nu)$ is constant for $\nu \geq r$.
\end{proof}

\begin{prop} \label{prop:gcd-Sk-m}
Let $k, m$ be positive integers with even $k$. Then
\[
   (S_k(m),m) = \frac{m}{(D_k,m)} \qquad \text{and} \qquad
     \min_{m \, \geq \, 1} \, g_k(m) = \frac{1}{D_k}.
\]
\end{prop}

\begin{proof}
Let $m > 1$, since the case $m = 1$ is trivial.
By Theorem~\ref{thm:congr-sk-bk} we have
\[
   S_k(m) \equiv \frac{N_k}{D_k} \, m \pmod{m}.
\]
For each prime power $p^{e_p} \pdiveq m$, we then infer that
$p^{e_p} \pdiv S_k(m)$, if $p \notdiv D_k$; otherwise
$p^{e_p-1} \pdiveq S_k(m)$, since $D_k$ is square-free due
to \eqref{eq:Bk-denom}. This gives the first equation above.
Using Lemma \ref{lem:seq-gcd} and \eqref{eq:loc-gcd-Sk},
we deduce the relation
\[
   g_k(m) = \frac{(S_k(m),m^k)}{m} \geq
     \frac{(S_k(m),m)}{m} = \frac{1}{(D_k,m)}.
\]
If $m=D_k$, then we even have that $(S_k(m),m^\nu) = 1$ for
$\nu \geq 1$, giving the minimum with $g_k(m) = 1 / D_k$.
\end{proof}

\begin{prop} \label{prop:gcd-Sk-m2}
Let $k, m$ be positive integers with even $k$. Then
\[
   \frac{(S_k(m),m^2)}{m} = \frac{(N_k,m)}{(D_k,m)}.
\]
\end{prop}

\begin{proof}
The case $k=2$ follows by \eqref{eq:Sk-sum}, $B_2=\frac16$,
and $((m-1)(2m-1),m)=1$. Now let $k \geq 4$, $m \geq 2$,
and assume that $(D_k,m)=1$.
Applying Theorem~\ref{thm:congr-sk-bk} for this case we then have
\begin{equation} \label{eq:loc-Sk-frac-m2}
   S_k(m) \equiv \frac{N_k}{D_k} \, m \pmod{m^2}.
\end{equation}
Thus we deduce that $(S_k(m),m^2) = m \, (N_k,m)$.
Now let $m$ be arbitrary. Using Proposition~\ref{prop:gcd-Sk-m}
we obtain the relation
\[
   (S_k(m),m^2) = c_{k,m} (S_k(m),m) = c_{k,m} \frac{m}{(D_k,m)}
\]
with some integer $c_{k,m} \geq 1$. Since $(N_k,D_k)=1$, those factors of
$(N_k,m)$ can only give a contribution to the factor $c_{k,m}$; while other
factors of $m$ are reduced by $(D_k,m)$.
To be more precise, consider a prime $p$ where $p^r \pdiveq m$:
If $p \pdiv D_k$, then $\ord_p \, (S_k(m),m^\nu) = r-1$
for $\nu \geq 1$ by Proposition \ref{prop:gcd-Sk-m} and Lemma \ref{lem:seq-gcd}.
Otherwise $p \notdiv D_k$ and \eqref{eq:loc-Sk-frac-m2} remains
valid$\pmod{p^{2r}}$ by Theorem~\ref{thm:congr-sk-bk}.
Hence $c_{k,m} = (N_k,m)$, which yields the result.
\end{proof}

\begin{prop} \label{prop:gcd-Sk-m3}
Let $k, m$ be positive integers with even $k$. Then
\[
   \frac{(S_k(m),m^3)}{m} = \frac{(N_k,m^2)}{(D_k,m)}.
\]
\end{prop}

\begin{proof}
The cases $k=2,4,6,8$ are compatible with Proposition~\ref{prop:gcd-Sk-m2},
since $|N_k| = 1$. Now let $k \geq 10$, $m \geq 2$,
and assume that $m \pdiv N_k$.
Using Theorem~\ref{thm:congr-sk-bk} we have for this case that
\begin{equation} \label{eq:loc-Sk-frac-m3}
   S_k(m) \equiv \frac{N_k}{D_k} \, m \pmod{m^3}.
\end{equation}
This shows that $(S_k(m),m^3) = m \, (N_k,m^2)$.
Now let $m$ be arbitrary. With Proposition~\ref{prop:gcd-Sk-m2}
we obtain the relation
\[
   (S_k(m),m^3) = d_{k,m} \, (S_k(m),m^2)
     = d_{k,m} \, m \frac{(N_k,m)}{(D_k,m)}
\]
with some integer $d_{k,m} \geq 1$.
Consider a prime $p$ where $p^r \pdiveq m$:
If $p \notdiv N_k$, then
\[
   \ord_p \, ( S_k(m), m^\nu ) \leq r, \quad \nu \geq 1,
\]
using Propositions \ref{prop:gcd-Sk-m} and \ref{prop:gcd-Sk-m2}
and Lemma~\ref{lem:seq-gcd}. Thus $p$ gives no contribution to $d_{k,m}$.
If $p \pdiv N_k$, then \eqref{eq:loc-Sk-frac-m2} and \eqref{eq:loc-Sk-frac-m3}
remain valid$\pmod{p^{2r}}$ and$\pmod{p^{3r}}$ by Theorem~\ref{thm:congr-sk-bk},
respectively. So a power of $p$ gives a contribution to
$d_{k,m}$. Counting the prime powers, which fulfill both
\eqref{eq:loc-Sk-frac-m2} and \eqref{eq:loc-Sk-frac-m3},
we then finally deduce that $d_{k,m} = (N_k,m^2)/(N_k,m)$.
\end{proof}

\begin{corl} \label{corl:gcd-Sk-mk}
Let $k, m$ be positive integers with even $k$. Then
\[
   (S_k(m),m^k) = e_{k,m} (S_k(m),m^3),
\]
where $e_{k,m}$ is a positive integer with the property
that $p \pdiv e_{k,m}$ implies that $p \pdiv N_k$.
\end{corl}

\begin{proof}
As in the proof of Proposition~\ref{prop:gcd-Sk-m3}, we can use the same
arguments. A prime $p$ with $p \notdiv N_k$ cannot give a contribution
to $e_{k,m}$ anymore.
\end{proof}

\begin{proof}[Proof of Theorem \ref{thm:Sk-gcd}]
The minimum of $g_k$ is shown by Proposition~\ref{prop:gcd-Sk-m}.
As a consequence of Proposition~\ref{prop:gcd-Sk-m3} and
Corollary~\ref{corl:gcd-Sk-mk}, it follows for arbitrary $m$ that
$g_k(m)=1$ if and only if $(D_k N_k,m)=1$.
Combining Propositions \ref{prop:gcd-Sk-m} -- \ref{prop:gcd-Sk-m3}
we have achieved that
\begin{equation} \label{eq:loc-Sk-mu}
   (S_k(m),m^\nu) = m \frac{(N_k,m^{\nu-1})}{(D_k,m)}, \quad \nu = 1,2,3.
\end{equation}
The values of $g_k(m)$ for $m=D_k, |N_k|, D_k |N_k|$ follow easily
by \eqref{eq:loc-Sk-mu} using Lemma~\ref{lem:seq-gcd},
since $(S_k(m),m^\nu)$ is constant for $\nu \geq 2$ in these cases.
If $(D_k,m)=1$ and $|N_k| \pdiv m$, then $g_k(m) = |N_k|$ by the same
arguments, which implies that
\begin{equation} \label{eq:loc-max-gk-3}
   \max_{m \, \geq \, 1} \, g_k(m) \geq |N_k|.
\end{equation}
It remains the case where $N_k$ is square-free. By \eqref{eq:loc-Sk-mu}
and Lemma~\ref{lem:seq-gcd} we conclude that $(S_k(m),m^\nu)$ is
constant for $\nu \geq 2$ for arbitrary $m$. Thus $g_k(m)=(N_k,m)/(D_k,m)$
in this case. Consequently \eqref{eq:loc-max-gk-3} holds with equality.
\end{proof}

\section*{Acknowledgement}
The author wishes to thank both the Max Planck Institute for Mathematics at
Bonn for an invitation for a talk in February 2010 and especially Pieter Moree
for the organization and discussions on the Erd\H{o}s-Moser equation.

\bibliographystyle{amsplain}

\end{document}